\documentclass[11pt]{article}

\usepackage{setspace}

\usepackage{amsmath, amsthm, amsfonts, amssymb, eufrak, makeidx, multicol}

\newtheorem{prop}{Proposition}[section]
\newtheorem{thm}[prop]{Theorem}
\newtheorem{lemma}[prop]{Lemma}
\newtheorem{cor}[prop]{Corollary}
\theoremstyle{definition} 
\newtheorem{defn}[prop]{Definition}
\newtheorem{example}[prop]{Example}

\newtheorem{question}[prop]{Question}
\newtheorem{rem}[prop]{Remark}

\DeclareMathSymbol{\subsetneq}{\mathord}{AMSb}{"26}

\newcommand{\mb}{\mathbb}

\newcommand{\p}{\mathfrak{p}}

\renewcommand{\a}{\mathfrak{a}}

\newcommand{\C}{\mb{C}}

\newcommand{\F}{\mb{F}}

\newcommand{\M}{\operatorname{M}}

\newcommand{\N}{\mb{N}}

\newcommand{\TA}{\operatorname{TA}}
\newcommand{\E}{\operatorname{E}}
\newcommand{\EA}{\operatorname{EA}}
\newcommand{\GA}{\operatorname{GA}}

\newcommand{\BA}{\operatorname{BA}}

\newcommand{\GL}{\operatorname{GL}}

\newcommand{\J}{\operatorname{J}}
\newcommand{\Ker}{\operatorname{Ker}}
\newcommand{\Char}{\operatorname{char}}

\newcommand{\ov}{\overline}

\title{ \bf Linearized polynomial maps \\ over finite fields}
\author{Joost Berson}

    \date{}

\begin{document}

\maketitle

\abstract{ \noindent We consider polynomial maps described by
so-called {\it (multivariate) linearized polynomials}. These
polynomials are defined using a fixed prime power, say $q$.
Linearized polynomials have no mixed terms. Considering invertible
polynomial maps without mixed terms over a characteristic zero
field, we will only obtain (up to a linear transformation of the
variables) triangular maps, which are the most basic examples of
polynomial automorphisms. However, over the finite field $\F_q$
automorphisms defined by linearized polynomials have (in general) an
entirely different structure. Namely, we will show that the
linearized polynomial maps over $\F_q$ are in one-to-one
correspondence with matrices having coefficients in a univariate
polynomial ring over $\F_q$. Furthermore, composition of polynomial
maps translates to matrix multiplication, implying that invertible
linearized polynomial maps correspond to invertible matrices.

This alternate description of the linearized polynomial automorphism
subgroup leads to the solution of many famous conjectures (most
notably, the Jacobian Conjecture) for this kind of polynomials and polynomial maps.  \\
}


\footnote{Funded by a Free Competition grant from the Netherlands Organisation for Scientific Research (NWO)\\}
\footnote{Joost Berson, Radboud University, Faculty of Science, P.O. Box 9010, 6500 GL Nijmegen,}
\footnote{The Netherlands, j.berson@science.ru.nl}

\section{Introduction}

Let $K[X]:=K[X_1,\ldots\!,X_n]$ be a polynomial ring over a field $K$. 
A natural problem in commutative algebra and algebraic geometry is 
to understand the group $\GA_n(K)$ of automorphisms of $K[X]$ preserving $K$. 
There are various long-standing open problems and conjectures in 
affine algebraic geometry concerning polynomial rings and their automorphisms 
(see~\cite{Essen}, \cite{E-R} and \cite{Kraft} for more details). 
Below we mention a few of the most famous ones. (Precise definitions 
will be provided in later sections.)

Polynomial automorphisms are generally studied over a field of
characteristic zero, but the prime characteristic case is gaining
interest (for example in \cite{Adjam},\cite{Borisov},\cite{DYu09},\cite{Mau01} and \cite{Nous}). 
In Section~\ref{solutions} of this paper, for the problems
and conjectures mentioned below, we give a complete answer in cases 
involving {\it linearized polynomials} (over a finite field $\F_q$), 
the main objects of interest of this paper. These polynomials, 
which are by definition (Section~\ref{PPP}) $\F_q$-linear combinations 
of monomials of the form $X_i^{{}^{q^m}}$, have thus far only been studied in case $n=1$, 
first by Ore in~\cite{Ore1} and~\cite{Ore2} (more on that in the same section). 
Section~\ref{PPP} is also devoted to a proof of the fact that the linearized polynomial maps
over $\F_q$ are in one-to-one correspondence with matrices having
coefficients in a univariate polynomial ring over $\F_q$ (where
$\F_q$ is the finite field with $q$ elements).

Finally, in Section~\ref{nomixedterms}, we will emphasize the
exceptional nature of linearized polynomial maps over finite fields. 
Namely, these maps form a special example of polynomial maps 
without mixed terms, which can be studied over a field of any 
characteristic. But the main result of this last section is that, 
over a characteristic zero field, every automorphism defined by polynomials
without mixed terms is a triangular automorphism (after a linear
transformation of the variables). This is certainly not the case for
linearized polynomial maps over a finite field. Also, the other 
problems and conjectures mentioned below are discussed for the case of 
polynomials without mixed terms. \\
\newline

\vspace{0.1cm}

\noindent \textbf{Tame Generators Problem}: Give necessary and
sufficient conditions for tameness of automorphisms of $K[X]$.\\
\newline
\noindent In two variables, this has already been solved by
Jung~\cite{Jung} and Van der Kulk~\cite{vdK}, saying that {\it all}
automorphisms in two variables are tame. In more variables there is
only one big result: Shestakov and Umirbaev gave a criterion for
tameness (over characteristic zero fields) of automorphisms of the
form $(f_1(X_1,X_2,X_3),f_2(X_1,X_2,X_3),X_3)$ in their
groundbreaking paper~\cite{S-U}. This gave a negative answer to the
question of tameness of the famous Nagata automorphism, introduced
in~\cite{Nag} (viewed as an automorphism in three variables over a field). 
We will show that all linearized polynomial automorphisms are tame,
in any dimension (Theorem~\ref{diagonalize}).\\
\newline

\vspace{0.05cm}

\noindent \textbf{Jacobian Conjecture}: If a polynomial map $f$ over a field
$K$ with $\Char(K)=0$ has invertible Jacobian matrix, then $f$ itself is invertible.\\
\newline
\noindent This famous conjecture was first proposed by
Keller~\cite{Keller} in 1939 for $K=\C$. After more than six decades
of intensive study by mathematicians, the conjecture is still open,
even for the case n = 2. It is listed as one of the 18 important
mathematical problems for the 21st century in Smale's
list~\cite{Smale}. More background and (references to) partial
results on the Jacobian Conjecture can be found in \cite{BCW}
and~\cite{Essen}. In nonzero characteristic the conjecture is 
easily shown to be false, but we will present an analogue of this conjecture 
for linearized polynomial maps, and give a proof (Corollary~\ref{JCPPP}).\\
\newline

\vspace{0.05cm}

\noindent \textbf{Coordinate Recognition Problem}: Given a polynomial $f\in
K[X]$, give necessary and sufficient conditions for $f$ to be a coordinate.\\
\newline
\noindent In case we have two variables, this problem has already
been solved in \cite{CK} and in \cite{lnd3}. The Coordinate Recognition Problem 
is still open for three or more variables. Our Proposition~\ref{unimodular} 
describes exactly when a linearized polynomial is a coordinate.\\
\newline

\vspace{0.05cm}

\noindent \textbf{Polynomial Ring Recognition Problem}: For a finitely
generated $K$-algebra $A$, give necessary and sufficient conditions
for $A$ to be (isomorphic to) a polynomial ring over $K$.\\
\newline
\noindent A necessary condition for being a polynomial ring over $K$
is that $A$ is a domain. Surprisingly, if $A$ is defined by
linearized polynomials this is also sufficient
(Corollary~\ref{PRRP}). This doesn't hold in general for an algebra 
defined over a field of characteristic zero,
where the problem has only been solved in case $A$ is at most
two-generated over $K$.
Corollary~\ref{PRRP} also implies the Abhyankar-Sathaye Conjecture
below, but then for linearized polynomials over a finite field
(Theorem~\ref{ASCq}).
In characteristic zero, this conjecture has only been completely solved for $n\leq2$.\\
\newline
\textbf{Abhyankar-Sathaye Conjecture}: If $\Char(K)=0$ and $f \in K[X_1,\ldots\!,X_n]$ 
satisfies $K[X_1,\ldots\!,X_n]/(f) \cong_KK[Y_1,\ldots\!,Y_{n-1}]$, then $f$ is a coordinate.\\
\newline

\noindent Last but not least, we present the\\
\newline
\textbf{Linearization Conjecture}: If an automorphism over a field
$K$ with $\Char(K)=0$ has finite order, then it is conjugate to a
linear automorphism.\\

\noindent An automorphism that is conjugate to a linear one is called
{\it linearizable}. For $n=2$ the (affirmative) answer easily follows from the
structure of $\GA_2(K)$, which was already observed in \cite{Kraft}.
For $n\geq3$ this conjecture is still unsolved.
However, we will show (Corollary~\ref{linearizable}) that a
linearized polynomial automorphism over $\F_q$ of finite order
relatively prime to $q$, is linearizable.\\

\section{Polynomial maps, conventions}

Associating a matrix to a polynomial map is a recurring thing in this paper, 
so first we write down the basic notations used in this paper concerning matrices. 
Given any commutative ring $R$, let $\M_{m\times n}(R)$ (or $\M_n(R)$, if $m=n$) 
be the set of all $m\times n$ matrices with entries in $R$.
For the group of all invertible matrices in $\M_n(R)$ we use the usual notation
$\GL_n(R)$. $I_n$ will be the identity matrix in $\GL_n(R)$. 

A \textit{polynomial map over $K$} is a list $f=(f_1,\ldots\!,f_m)$
of polynomials in $K[X]$. We can view polynomial maps as $K$-algebra
homomorphisms $K[Y]\to K[X]$, $Y_i\mapsto f_i$, where
$Y:=(Y_1,\ldots\!,Y_m)$ is another list of variables. But they are
often also identified with maps $K^n\to K^m$ given by polynomial
substitutions, which is actually only an exact identification if $K$
is infinite.

Now consider another polynomial map $g=(g_1,\ldots\!,g_n)$, with
each $g_i\in K[Z]$ for yet another list of variables
$Z=(Z_1,\ldots\!,Z_l)\,$. In the usual notation, the composition of
$f$ and $g$ is defined as $f\circ
g=(f_1(g_1,\ldots\!,g_n),\ldots\!,f_m(g_1,\ldots\!,g_n))$.
Restricting to the case $m=n$, the map $f$ is called an
\textit{invertible polynomial map} or \textit{automorphism} if there
exists another $g=(g_1,\ldots\!,g_n)\in K[X]^n$ with $f\circ
g=g\circ f=X$ (the identity map). Furthermore, we call a polynomial
in $K[X]$ a \textit{coordinate} if it equals one of the components
$f_i$ of some automorphism $f$.

The automorphisms form a group, $\GA_n(K)$. $\GL_n(K)$ is usually
viewed as a subgroup (the subgroup of {\it linear automorphisms}), 
but there are more ``usual'' subgroups. They will be introduced 
in this paper where they are needed. As the first and foremost example 
of associating a matrix to a polynomial map, we write $\J\!f$ for 
the Jacobian matrix $(\genfrac{}{}{}{1}{\partial f_i}{\partial X_j})$ 
of a polynomial map $f$. By the chain rule, for any automorphism $f$ 
we have $\J\!f\in \GL_n(K[X])$, whence $|\J\!f| \in K^*$. 
(Throughout this paper, the operator $|\cdot|$ takes the determinant of a matrix.)\\ \ \\

\section{Linearized polynomial maps and the $q$-Jacobian}\label{PPP}\

\noindent Here we will describe the main objects of study of this
paper, and their basic properties. For now, $X$ denotes just one variable.\\

\begin{defn}
Let $q$ be a positive power of a prime number. Then $\F_q[X]^{(q)}$
will be the $\F_q$-subspace of $\F_q[X]$ generated by all monomials
of the form $X^{{}^{q^m}}$ (with $m\geq0$). Furthermore, the {\it
composition} $f\circ g$ of $f,g\in\F_q[X]^{(q)}$ is defined as the
substitution of the two polynomials, {\it i.e.} $(f\circ
g)(X):=f(g(X))$.
\end{defn}\

\begin{rem}\label{corresp0}
The elements of $\F_q[X]^{(q)}$ are precisely the polynomials in
$\F_q[X]$ that induce an $\F_q$-linear map $K\to K$, where $K$ is any
infinite extension field of $\F_q$. Indeed, $X^q$ induces the $\F_q$-linear
map $x\mapsto x^q\ (x\in K)$, and any map induced by an element of
$\F_q[X]^{(q)}$ is an $\F_q$-linear combination of iterates of this
particular map. On the other hand, suppose $f\in\F_q[X]$ induces an
$\F_q$-linear map $K\to K$, and let $X^m$ be a monomial appearing in
$f$. Since $K$ is an infinite field, the hypothesis implies that
$f(X+Y)=f(X)+f(Y)$ and $f(aX)=af(X)$, where $Y$ is a new variable
and $a$ generates the multiplicative group of $\F_q$. Comparing terms 
of equal degree yields $(X+Y)^m=X^m+Y^m$ and $(aX)^m=aX^m$. 
Let $p$ be the unique prime number such that $q=p^r$, with $r\geq1$. 
Suppose $m$ is not a power of $p$, say $m=dp^e$ with $d>1$, $p\nmid d$ 
and $e\geq0$. Then $(X+Y)^{{}^m}=((X+Y)^{{}^{p^e}})^{{}^d}=(X^{{}^{p^e}}+Y^{{}^{p^e}})^{{}^d}$
contains the nonzero term $dX^{{}^{(d-1)p^e}}Y^{{}^{p^e}}$, which
contradicts the fact that $(X+Y)^m=X^m+Y^m$. Hence, $m$ is a power
of $p$. Since $a$ generates $\F_q^*$ and $a\in\F_m$ (as $a^m=a$), we
have $\F_q\subseteq\F_m$. So $\F_m$ is a finite dimensional
$\F_q$-space, whence $m$ is a power of $q$.
\end{rem}\

\noindent The above remark implies that $\F_q[X]^{(q)}\,$ is closed under composition.
Moreover, this composition operation has some remarkable properties 
compared to the composition of any two univariate polynomials over any field 
(which can be defined in a similar way). For one easily verifies that
\begin{itemize}
\item $f(g+h)=f(g)+f(h)\ \ \forall\ f,g,h \in \F_q[X]^{(q)}$
\item composition is commutative: $\ f(g(X))=g(f(X))\ \ \forall\ f,g \in \F_q[X]^{(q)}$
\end{itemize}
(The first property follows directly from Remark~\ref{corresp0}.)
Using these facts, it is easy to check that $\F_q[X]^{(q)}\,$ is a
commutative ring (with addition inherited from $\F_q[X]$, and
``multiplication'' being composition). Also, note that $X$ is the
identity element in this ring, and that $\F_q\to\F_q[X]^{(q)}\!$,
$\,a\mapsto aX$ makes $\F_q[X]^{(q)}$ an $\F_q$-algebra. In fact,
Theorem~\ref{corresp1} will show that $\F_q[X]^{(q)}\,$ is
isomorphic as $\F_q$-algebra to the univariate polynomial ring over
$\F_q$ !

Here we should remark that linearized polynomials (sometimes referred to
as ``$p$-polynomials'' or ``$q$-polynomials'') have already been studied in several
papers. Their focus is mostly on the fact that the roots of a linearized polynomial
form an $\F_q$-subspace of its splitting field (the kernel of the induced linear map).
The result of Theorem~\ref{corresp1} was first mentioned by Ore (\cite{Ore1},\cite{Ore2}).
Later, the property mentioned in Remark~\ref{corresp0} was noted in~\cite{Berlekamp}
and~\cite{Jamison}. Both properties also appeared in~\cite{Baker},\cite{Lidl} and~\cite{MWSloane}.
However, in this section we will also define {\it multivariate} linearized
polynomials (the main objects of study of this paper), which have not been studied before in the literature.\\

\begin{thm}\label{corresp1}
There is a unique isomorphism of $\F_q$-algebras $\delta:\
\F_q[X]^{(q)}\, \rightarrow \F_q[t]$ such that
$\delta\Big(X^{{}^{q^m}}\Big)=\,t^m$ for all $m\geq0$. Thus,
$\delta(f(g))=\delta(f)\cdot\delta(g)\ \ \forall\ f,g \in
\F_q[X]^{(q)}$.
\end{thm}

\begin{proof}
By the universal property of $\F_q$-algebras, there is a unique
$\F_q$-algebra homomorphism $\F_q[t]\to\F_q[X]^{(q)}$ such that
$t\mapsto X^q$. This map clearly gives a one-to-one correspondence
between the $\F_q$-bases $\{t^m\,|\,m\geq0\}$ and
$\{X^{{}^{q^m}}|\,m\geq0\}$. Hence, the algebra homomorphism is a
vector space isomorphism, and thus even an $\F_q$-algebra
isomorphism (with inverse $\delta$).
\end{proof}\

\noindent Now let $X:=(X_1,\ldots\!,X_n)$ be a list of variables. Then
the polynomials in
$$
\F_q[X]^{(q)}:=\F_q[X_1]^{(q)}\,\oplus\,\cdots\,\oplus\,\F_q[X_n]^{(q)}
$$
are called \emph{(multivariate) linearized polynomials in $X_1,\ldots\!,X_n$}. 
And the elements of $(\F_q[X]^{(q)})^m$ (as subset of $\F_q[X]^m$) are 
the \emph{(multivariate) linearized polynomial maps}.\\

\begin{rem}
 The elements of $\F_q[X]^{(q)}$ are precisely the polynomials in
$\F_q[X]$ that induce an $\F_q$-linear map $K^n\to K$, where $K$ is any
infinite extension field of $\F_q$. Indeed, any term of a given element of
$\F_q[X]^{(q)}$ is in fact an element of $\F_q[X_i]^{(q)}$ for some $i$
(and the induced map $K^n\to K$ factorizes through the projection $K^n\to K$
on the $i$th factor), so Remark~\ref{corresp0} implies that elements
of $\F_q[X]^{(q)}$ induce $\F_q$-linear maps. On the other hand,
suppose $f\in\F_q[X]$ induces an $\F_q$-linear map $K^n\to K$,
and let $X_1^{m_1}\cdots X_n^{m_n}$ be a monomial appearing in $f$.
As $K$ is an infinite field, the hypothesis implies that
$f(X+Y)=f(X)+f(Y)$, where $Y:=(Y_1,\ldots\!,Y_n)$ is a new
list of variables. But then
\begin{equation}\label{multilin}
(X_1+Y_1)^{m_1}\cdots(X_n+Y_n)^{m_n}=X_1^{m_1}\cdots X_n^{m_n}+Y_1^{m_1}\cdots Y_n^{m_n}
\end{equation}
since the lefthandside exactly contains all terms 
$X_1^{\alpha_1}Y_1^{\beta_1}\cdots X_n^{\alpha_n}Y_n^{\beta_n}$ in $f(X+Y)$ 
such that $\alpha_i+\beta_i=m_i$ for all $i$. Now suppose we have $i\neq j$ 
such that both $m_i>0$ and $m_j>0$. Substituting $X_i=Y_j=0$ in~(\ref{multilin}), we get that
$$
Y_i^{m_i}X_j^{m_j}\prod_{k\neq i,j}(X_k+Y_k)^{m_k} = 0
$$
which is a contradiction. Thus, only one of the $m_i$ is positive, {\it i.e.}
the monomial under consideration is a power of one of the $X_i$. As a result,
$f=f_1+\cdots+f_n$ with $f_i\in K[X_i]$ for all $i$. We may even assume that
$f_1(0)=\cdots=f_n(0)=0$, since the hypothesis on $f$ implies that it has no
constant term. Then each $f_i$ induces an $\F_q$-linear map $K\to K$ 
(the composition of $f$ and the embedding $K\to K^n$, $\alpha\mapsto\alpha e_i$).
Remark~\ref{corresp0} now implies that $f\in\F_q[X]^{(q)}$.
\end{rem}\

\noindent The composition of linearized polynomial maps gives another one: if $Z:=(Z_1,\ldots\!,Z_l)$
is another list of variables, then the composition (already defined for polynomial maps in general)
of $f=(f_1,\ldots\!,f_m)\in(\F_q[X]^{(q)})^m$ and $g=(g_1,\ldots\!,g_n)\in(\F_q[Z]^{(q)})^n$ is the element
$f\circ g=(f_1(g_1,\ldots\!,g_n),\ldots\!,f_m(g_1,\ldots\!,g_n))$, which can easily be shown to be
an element of $(\F_q[Z]^{(q)})^m$. For the case $m=n$ this implies that $(\F_q[X]^{(q)})^n$ 
is closed under composition. Also, Theorem~\ref{diagonalize} will show that 
$\GA_n(\F_q)^{(q)}:=\GA_n(\F_q)\cap(\F_q[X]^{(q)})^n$ is a subgroup of $\GA_n(\F_q)$. 

Theorem~\ref{corresp2} will show that we can view polynomial maps in $(\F_q[X]^{(q)})^m$
as matrices having univariate polynomials over $\F_q$ as entries. To make this explicit,
we define the $q$-Jacobian of polynomial maps of this form. The definition is based on
certain maps $\delta_j$ (one for each variable $X_j$) that are very similar to the map
$\delta$ of Theorem~\ref{corresp1}.\\

\begin{defn} \label{J_q}
Let $f=(f_1,\ldots\!,f_m)\in(\F_q[X]^{(q)})^m$, and $t$ a new variable.
For each $j \in \{1,\ldots\!,n\}$, let $\delta_j:\ \F_q[X]^{(q)}\, \rightarrow \F_q[t]$ be
the $\F_q$-linear map uniquely determined by
$$
\delta_j\left(X_i^{{}^{q^m}}\right) = \left\{\begin{array}{ll} t^m & i=j\\ 0 & i\neq j\end{array}\right.\ \ \ \ (m=0,1,2,\ldots\,)
$$
Furthermore, we define $\J_q(f)$ as the matrix $(\delta_j(f_i))\in\M_{m\times n}(\F_q[t])$, 
and call it the {\it $q$-Jacobian} of $f$ (or ``{\it J-$q$-bian}'').
\end{defn}\

\begin{rem}\label{1to1}
The map
\begin{eqnarray*}
(\F_q[X]^{(q)})^m & \longrightarrow & \M_{m\times n}(\F_q[t])\\
(f_1,\ldots\!,f_m) & \mapsto & \J_q(f)
\end{eqnarray*}
is obviously one-to-one and onto. We will need this fact henceforth.
\end{rem}\

\noindent Now let $g=(g_1,\ldots\!,g_n)\in(\F_q[Z]^{(q)})^n$.
We will denote the maps $\F_q[Z]^{(q)}\, \rightarrow \F_q[t]$
(similarly defined as the $\delta_j$) by $\varepsilon_j$. In this situation we have\\

\begin{thm} \label{corresp2}
If $f=(f_1,\ldots\!,f_m)\in(\F_q[X]^{(q)})^m$ and $g=(g_1,\ldots\!,g_n)\in(\F_q[Z]^{(q)})^n$,
then $\J_q(f\circ g)=\J_q(f)\J_q(g)$.

In particular, $\J_q$ induces an isomorphism of $\,\F_q$-algebras $\,(\F_q[X]^{(q)})^n \overset\sim\longrightarrow \M_n(\F_q[t])$.
\end{thm}

\begin{proof}
Write $f_i=\sum_{k=1}^n f_i^{\scriptscriptstyle (k)}(X_k)$ and
$g_i=\sum_{r=1}^l g_i^{\scriptscriptstyle (l)}(Z_r)$ for all $i$. Then
$$
f_i(g) = \sum_{k=1}^n f_i^{\scriptscriptstyle (k)}(g_k) =
\sum_{k=1}^n f_i^{\scriptscriptstyle (k)}(\sum_{r=1}^l
g_k^{\scriptscriptstyle (r)}(Z_r)) = \sum_{r=1}^l \sum_{k=1}^n
f_i^{\scriptscriptstyle (k)}(g_k^{\scriptscriptstyle (r)}(Z_r))
$$
and thus the $(i,j)$-entry of $\J_q(f(g))$ equals
\begin{eqnarray*}
\varepsilon_j(f_i(g))=\varepsilon_j\,(\sum_{k=1}^n f_i^{\scriptscriptstyle
(k)}(g_k^{\scriptscriptstyle (j)}(Z_j))) & = &
\sum_{k=1}^n\,\varepsilon_j\,(f_i^{\scriptscriptstyle
(k)}(Z_j))\cdot\varepsilon_j\,(g_k^{\scriptscriptstyle (j)}(Z_j))\\
 & = & \sum_{k=1}^n\,\delta_k\,(f_i^{\scriptscriptstyle
(k)}(X_k))\cdot\varepsilon_j\,(g_k)\\
 & = & \sum_{k=1}^n\,\delta_k\,(f_i)\cdot\varepsilon_j\,(g_k)
\end{eqnarray*}
which is exactly equal to the $(i,j)$-entry of the product
$(\delta_j(f_i))\cdot(\varepsilon_j(g_i))\,$. Thus,
$\J_q(f(g))=\J_q(f)\cdot\J_q(g)$. The second statement follows from Remark~\ref{1to1}.
\end{proof}\

\section{The famous problems and conjectures for linearized polynomials}\label{solutions}

\noindent This section is devoted to the solutions that we found for 
the famous problems and conjectures that were stated in the Introduction, 
for the cases where the involved polynomials are linearized polynomials.\\

\subsection{Tame Generators Problem and Jacobian Conjecture}\

\noindent Before solving the Tame Generators Problem for linearized polynomial maps,
we recall the concept of tameness.\\

\begin{defn}
$\EA_n(K)$ (for any field $K$) is the subgroup of $\GA_n(K)$
generated by the elementary automorphisms.  An {\it elementary}
automorphism is one of the form $(X_1,\ldots\!,X_{i-1},X_i+f_i,X_{i+1},\ldots\!,X_n)$
for some $i$, where $f_i \in K[\hat{X}_i]$. Furthermore, $\TA_n(K)$, the
group of {\it tame} automorphisms, is the subgroup generated by
$\GL_n(K)$ and $\EA_n(K)$.
\end{defn}\

\noindent As mentioned in the Introduction, the question which automorphisms 
are tame is still open in general if $n\geq3$. However, Theorem~\ref{diagonalize} 
will show that all invertible linearized polynomial maps are tame. 
To formulate the precise statement, we need to define a few automorphism subgroups 
consisting of linearized polynomial maps. First, we put
$$
\EA_n(\F_q)^{(q)}:=\langle\,(X_1,\ldots\!,X_{i-1},X_i+f_i,X_{i+1},\ldots\!,X_n)\,|\,1\leq i\leq n, f_i \in \F_q[\hat{X}_i]^{(q)}\,\rangle
$$
Furthermore, let $\TA_n(\F_q)^{(q)}:=\langle\,\EA_n(\F_q)^{(q)}\,,\GL_n(\F_q)\,\rangle$.
Under the isomorphism of Theorem~\ref{corresp2}, the subgroup $\EA_n(\F_q)^{(q)}$
corresponds to $\E_n(\F_q[t])$, the subgroup of $\GL_n(\F_q[t])$
generated by all elementary matrices. Also, this isomorphism is the identity
on $\GL_n(\F_q)$. \\

\begin{thm} \label{diagonalize}
Let $f=(f_1,\ldots\!,f_m) \in (\F_q[X]^{(q)})^m$. Then there exist
$h_1 \in \TA_m(\F_q)^{(q)}$ and $h_2 \in \TA_n(\F_q)^{(q)}$ such
that $h_1fh_2$ is a ``diagonal map'', \emph{i.e.} a map of the form
$g=(g_1,\ldots\!,g_m)$, where $g_i \in \F_q[X_i]^{(q)}$ for all $i$
(and $g_i=0$ if $m>n$ and $n<i\leq m$).

Furthermore, $\GA_n(\F_q)^{(q)}=\TA_n(\F_q)^{(q)}$.
\end{thm}

\begin{proof}
$\J_q(f)$ is a matrix over a Euclidean domain, so there exist $M \in
\GL_m(\F_q[t])$
$(\,=\nolinebreak[4]\!\langle\E_m(\F_q[t]),\GL_m(\F_q)\rangle\,)$
and $N \in \GL_n(\F_q[t])$ such that $M\J_q(f)N$ is a (in general
non-square) diagonal matrix. By Remark~\ref{1to1}, there exist $h_1 \in
\TA_m(\F_q)^{(q)}$ and $h_2 \in \TA_n(\F_q)^{(q)}$ such that
$h_1fh_2$ is of the prescribed form. For the next statement, suppose
$f \in \GA_n(\F_q)^{(q)}$. The above says that $f$ is tamely equivalent to a map
$g=(g_1,\ldots\!,g_n)$ with $g_i \in \F_q[X_i]^{(q)}\,$ for all $i$.
Since $f$ is an automorphism, $g$ is too, so let
$h=(h_1,\ldots\!,h_n) \in \GA_n(\F_q)$ be the inverse of $g$. Since
$\F_q[X_1,\ldots\!,X_{i-1},X_{i+1},\ldots\!,X_n]$ is a domain, the
equations $g_i(h_i(X))=X_i$ imply that $h_i\in\F_q[X_i]$ and that both $g_i$ and 
$h_i$ have degree 1 (for all $i$). Consequently, $f \in \TA_n(\F_q)^{(q)}$.
\end{proof}\

\noindent Note that this theorem in particular implies that 
$f^{-1} \in \GA_n(\F_q)^{(q)}$ if $f \in \GA_n(\F_q)^{(q)}$, {\it i.e.} 
$\GA_n(\F_q)^{(q)}$ is a subgroup of $\GA_n(\F_q)$. As a result, 
we can affirm an analogue of the Jacobian Conjecture 
for linearized polynomial maps and their $q$-Jacobians.\\

\begin{cor}\label{JCPPP}
$f \in (\F_q[X]^{(q)})^n$ is an automorphism if and only if $\J_q(f) \in \GL_n(\F_q[t])$.
\end{cor}

\begin{proof}
This follows from Theorem~\ref{corresp2} and Theorem~\ref{diagonalize}: 
if $f \in \GA_n(\F_q)^{(q)}$ then $f^{-1} \in \GA_n(\F_q)^{(q)}$ and 
$\J_q(f)\J_q(f^{-1})=\J_q(ff^{-1})=I_n$. If on the other hand $\J_q(f) \in \GL_n(\F_q[t])$, 
let $g\in(\F_q[X]^{(q)})^n$ such that $\J_q(g)=(\J_q(f))^{-1}$ 
(which exists by Remark~\ref{1to1}). Then $\J_q(fg)=\J_q(f)\J_q(g)=I_n$ 
implies that $f$ is an automorphism with inverse $g$. 
\end{proof}

\noindent Note that if we take the {\it usual} Jacobian, the statement
doesn't hold; namely, the Jacobian of any linearized polynomial map
equals the Jacobian of its linear part.\\

\subsection{Coordinate Recognition Problem}\

\noindent Corollary~\ref{JCPPP} provides us with the following useful tool:
a criterion to decide whether a linearized polynomial is a coordinate.\\

\begin{prop} \label{unimodular}
For $f_1 \in \F_q[X]^{(q)}\,$, the following are equivalent.
\begin{enumerate}
\item $f_1$ is a coordinate of an automorphism in $\F_q[X]^n$
\item $f_1$ is a coordinate of an automorphism in $(\F_q[X]^{(q)})^n$
\item $(\delta_1(f_1),\ldots\!,\delta_n(f_1))=(1)$ in $\F_q[t]$
\end{enumerate}
\end{prop}

\begin{proof}[Proof of the equivalence of {\it 2.}\,and {\it 3.}]
$f_1$ is a coordinate in $(\F_q[X]^{(q)})^n$ if and only if $(\delta_1(f_1),\ldots\!,\delta_n(f_1))$
is a row that is extendible to a matrix in $\GL_n(\F_q[t])$ if and only if
$(\delta_1(f_1),\ldots\!,\delta_n(f_1))=(1)$ in $\F_q[t]$. (We use Remark~\ref{1to1} again.)
\end{proof}\

\noindent From this we obtain the remarkable fact (Corollary~\ref{polycoord})
that all prime power polynomials are essentially univariate ({\it i.e.},
up to a polynomial transformation). This fact in turn will help us complete the proof
of Proposition~\ref{unimodular}. \\

\begin{cor} \label{polycoord}
Every element of $\F_q[X]^{(q)}$ is a linearized polynomial in a coordinate of an automorphism in $(\F_q[X]^{(q)})^n$.
\end{cor}

\begin{proof}
Let $f \in \F_q[X]^{(q)}$, and $h:=\gcd(\delta_1(f),\ldots\!,\delta_n(f)) \in \F_q[t]$
(unique if we assume $h$ to be a monic polynomial).
Then we have $(\delta_1(f),\ldots\!,\delta_n(f))=(h)$ (as ideals in $\F_q[t]$),
and we can write $\delta_i(f)=hg_i$ with $g_1,\ldots\!,g_n \in \F_q[t]$.
Now let $\tilde{f_i} \in \F_q[X_i]^{(q)}\ (i=1,\ldots\!,n)$ and
$\tilde{h} \in \F_q[X_1]^{(q)}$ such that $\delta_i(\tilde{f_i})=g_i(t)$
and $\J_q(\tilde{h})=h(t)$ (using Remark~\ref{1to1} again). Then $\tilde{f}:=\tilde{f_1}+\cdots+\tilde{f_n}$ gives
$$
\J_q(f)=(\delta_1(f)\ \cdots\ \delta_n(f))=(hg_1\ \cdots\ hg_n)=h\cdot(g_1\ \cdots\ g_n)=
\J_q(\tilde{h})\J_q(\tilde{f})=\J_q(\tilde{h}(\tilde{f}))
$$
Hence, $f=\tilde{h}(\tilde{f})$, and
$(\delta_1(\tilde{f}),\ldots\!,\delta_n(\tilde{f}))=(g_1,\ldots\!,g_n)=(1)$ in $\F_q[t]$,
so $\tilde{f}$ is a coordinate by the equivalence of {\it 2.} and {\it 3.} in Proposition~\ref{unimodular}.
\end{proof}\

\begin{proof}[Proof of the equivalence of {\it 1.}\,and {\it 2.}(Proposition~\ref{unimodular})]
The only nontrivial implication is {\it 1.}$\,\Rightarrow\,${\it 2.}, so assume that
$f_1$ is a coordinate of an automorphism in $\F_q[X]^n$. By Corollary~\ref{polycoord},
$f_1=g_1(h_1)$ with $g_1\in\F_q[X_1]^{(q)}$ and $h_1\in\F_q[X]^{(q)}$, and such
that $h_1$ is the first coordinate of an automorphism in $(\F_q[X]^{(q)})^n$.
Applying the inverse of this automorphism to $f_1$, we deduce that
$g_1(X_1)$ is a coordinate as well. Just as in the proof of Theorem~\ref{diagonalize}, 
this implies that $g_1$ has degree 1, say $g_1(X_1)=aX_1+b$ with $a\in\F_q^{\,*}$ and $b\in\F_q$. 
But $g_1\in\F_q[X_1]^{(q)}$, so $b=0$. 
Now $f_1=ah_1$ is the first coordinate of an automorphism in $(\F_q[X]^{(q)})^n$.
\end{proof}\

\noindent One can write down many coordinates over finite fields of {\it such} a form,
that they can't possibly be coordinates when considered over a field of characteristic zero.
This is illustrated in the following example. A polynomial as described there, {\it i.e.}
of the form $\tilde{f}:=f(X)+Y^{q^n}$, can only be a coordinate over a characteristic zero field $K$
in the trivial cases $n=0$ or $f$ has degree 1 (as will follow from Proposition~\ref{coordrecognnonmixed}).\\

\begin{example}\label{nontrivial}
Any element of $\F_q[X,Y]^{(q)}$ (two variables) of the form $f(X)+Y^{q^n}$,
with $n\geq0$ and linear part of $f$ equal to $X$, is a coordinate.
Namely, let $g(X):=f(X)-X\in\F_q[X]^{(q)}$. Note that $g(X)=h(X)^q$ for some $h\in\F_q[X]^{(q)}$
(for $g$ contains no linear term), whence
$\hat{g}(t):=\delta_1(g(X))=\delta_1(X^q)\delta_1(h(X))=t\hat{h}(t)$, where $\hat{h}(t):=\delta_1(h(X))$. Thus,
$$
\left(\begin{array}{cc} 1+\hat{g}(t) & t^n\\
(-1)^{n+1}\hat{h}(t)^n &
\tfrac{1-(-\hat{g}(t))^n}{1-(-\hat{g}(t))}\end{array}\right) \in
\GL_2(\F_q[t])
$$
Note that the lower right entry is indeed an element of $\F_q[t]$: it equals 
the finite geometric series $1-\hat{g}(t)+\hat{g}(t)^2-\cdots+(-1)^{n-1}\hat{g}(t)^{n-1}$.
From the above we obtain
$$
(f(X)+Y^{q^n}, (-1)^{n+1}h(X)^{(n)}+\sum_{k=0}^{n-1}(-1)^kg(Y)^{(k)})
\in \GA_2(\F_q)
$$
where each exponent ``$(k)$'' of a polynomial denotes $k$-fold composition of that
polynomial with itself (and $g(Y)^{(0)}:=Y$). In particular, $h(X):=X^{q^{m-1}}$ ($m\geq1$) gives
$$
(X+X^{q^m}+Y^{q^n}, (-1)^{n+1}X^{q^{(m-1)n}}+\sum_{k=0}^{n-1}(-1)^kY^{q^{km}})
\in \GA_2(\F_q)
$$
Assuming $m,n\geq2$, and writing $n=rm+s$ with $r\in\N$ and $0\leq s\leq m-1$, we can also
complete this automorphism using a polynomial of lower degree. Namely,
$$
(X+X^{q^m}+Y^{q^n}, (-1)^rX^{q^{m-s}}+\sum_{k=0}^r(-1)^kY^{q^{km}})
\in \GA_2(\F_q)
$$
since
$$
\left(\begin{array}{cc} 1+t^m & t^n\\
(-1)^rt^{m-s} & \tfrac{1-(-t^m)^{r+1}}{1-(-t^m)}\end{array}\right)
\in \GL_2(\F_q[t])
$$
\end{example}\

\subsection{Polynomial Ring Recognition Problem and Abhyankar-Sathaye Conjecture}\

\noindent A finitely generated $K$-algebra $A$ can be represented as
$A=K[X]/I$, where $I$ is an ideal of $K[X]$. A necessary condition
for being a polynomial ring over $K$ is that $A$ is a domain, whence
$I$ must be a prime ideal. If $K$ is a finite field and $I$ is
generated by linearized polynomials, we will show that the condition
of being a domain is actually also sufficient
(Corollary~\ref{PRRP}). This differs significantly from the
characteristic zero case, which has only been solved in case $X$
represents at most two variables. We will first summarize the
results of this case.

To begin, if $I$ is even a maximal ideal, then $K[X]/I$ is a field, which is
of course only a polynomial ring over $K$ if it equals $K$ (the
units of both fields must coincide). In other words, the canonical
embedding $K\to K[X]/I$ is actually an isomorphism. In this case,
choosing $a_1,\ldots\!,a_n\in K$ such that $X_i-a_i\in I$ for all
$i$ (which exist since the embedding is onto),
we get that $I=(X_1-a_1,\ldots\!,X_n-a_n)$. So in case of a maximal ideal $I$,
$A$ is a polynomial ring if and only if $I$ is of this form.

This also solves the general case $n=1$, since any nonzero prime
ideal of $K[X]$ is then maximal. And in the case of two variables,
any non-maximal, nonzero prime ideal of $K[X]$ is generated by one
irreducible polynomial (since $K[X]$ is a factorial ring). Hence,
the following result, proved by Abhyankar and Moh in \cite{AM} and
independently by Suzuki in \cite{Suz}, completes the solution of the
two-variable Polynomial Ring Recognition Problem over a field of
characteristic zero.\\

\begin{thm}[Abhyankar-Moh-Suzuki]\label{AMS}
Let $K$ a field with $\Char(K)=0$. If a polynomial $f_1 \in K[X,Y]$
satisfies $K[X,Y]/(f_1) \cong_KK[Z]$, then $f_1$ is a coordinate.
\end{thm}\

\noindent Contrary to the characteristic zero case, several counterexamples 
to Theorem~\ref{AMS} have been found in characteristic $p>0$. 
Here is one which was also mentioned in~\cite{Nagata}.\\

\begin{example}\label{AMS-p}
Take any prime number $p>2$, and let $f_1:=Y^{p^2}-X^{2p}-X$. Then
$\F_p[X,Y]/(f_1) \cong \F_p[T]$, where $T$ is a variable; this
isomorphism is induced by $\varphi:\ \F_p[X,Y]\to \F_p[T], \ X
\mapsto T^{p^2}, Y \mapsto T^{2p}+T$. But we claim that $f_1$ is not a
coordinate.

First, note that $\varphi$ is indeed surjective since
$f_2:=Y-(Y^p-X^2)^2$ satisfies $\varphi(f_2)=T$. Now we show that
$\Ker(\varphi)=(f_1)$. Since the (Krull) dimensions of $\F_p[X,Y]$
and $\F_p[T]$ are equal to 2 resp.\,1, $\Ker(\varphi)$ must be a
height 1 prime ideal, and thus a principal ideal due to the
factoriality of $\F_p[X,Y]$. So it suffices to show that $f_1$ is
irreducible over $\F_p$. So let $\beta$ be an element of an
extension field of $K:=\F_p(X)$ such that
$\beta^{p^2}=\alpha:=X^{2p}+X$. Then
$f_1=Y^{p^2}-\alpha=(Y-\beta)^{p^2}$ over $K(\beta)$. Let $1\leq
m\leq p^2$ be minimal such that $(Y-\beta)^m\in K[Y]$. Then
$(Y-\beta)^m$ is irreducible over $K$, and in fact the only
irreducible factor of $f_1$ (since two positive powers of $Y-\beta$
cannot be coprime). Hence $f_1$ is a power of $(Y-\beta)^m$, and
$m\mid p^2$. Now suppose $m=1$ or $m=p$. Then
$Y^p-\beta^p=(Y-\beta)^p\in K[Y]$, so $\alpha=(\beta^p)^p\in
K^p=\F_p(X^p)$, a contradiction. As a result, $m=p^2$, and the
conclusion is that $f_1$ is irreducible over $K$.

Now suppose $f_1$ is a coordinate. Using the fact that $\F_p$ is a
field, Corollary~5.1.6 in \cite{Essen} yields an $f_2' \in \F_p[X,Y]$ 
with $\deg(f_2')<\deg(f_1)$ and $(f_1,f_2') \in\GA_2(\F_p)$. 
(Here ``$\deg$'' denotes the (total) degree of a polynomial.)
Since $\F_p[X,Y]/(f_1)=\F_p[f_1,f_2']/(f_1)=\F_p[\ov{f_2'}]$ 
($\ov{f_2'}$ being the equivalence class of $f_2'$ modulo $(f_1)$), 
we must have $\F_p[\varphi(f_2')]=\F_p[T]$, whence $\varphi(f_2')=aT+b$ with $a
\in \F_p^*,b \in \F_p$. But $\varphi(f_2)=T$, which implies that
$f_2'-af_2-b \in \Ker(\varphi)=(f_1)$. From
$\deg(f_2'-af_2-b)<\deg(f_1)$ we now conclude that $f_2'-af_2-b=0$.
Thus, $(f_1,f_2) \in \GA_2(\F_p)$. But according to Corollary~5.1.6 in
\cite{Essen} either $\deg(f_1)\mid\deg(f_2)$ or $\deg(f_2)\mid\deg(f_1)$, 
contradicting the fact that $\deg(f_1)=p^2$ and $\deg(f_2)=2p$.
\end{example}\

\noindent Theorem~\ref{PPP ideal} is the key to the solution of the
Polynomial Ring Recognition Problem for $\F_q$-algebras which are
defined by linearized polynomials (Corollary~\ref{PRRP}).\\

\begin{thm} \label{PPP ideal}
Let $\p$ be a prime ideal in $\F_q[X]$ generated by linearized polynomials.
Then these polynomials can be chosen in such a way that together they are extendible
to an automorphism in $\GA_n(\F_q)^{(q)}$.

More generally, let $\a$ be any ideal in $\F_q[X]$ generated by linearized polynomials.
Then there exist $h=(h_1,\ldots\!,h_n)\in\GA_n(\F_q)^{(q)}$, $r\leq n$ and
$g_i\in\F_q[X_i]^{(q)}\backslash\{0\}$ for $i=1,\ldots\!,r$,
such that $\a=(g_1(h_1),\ldots\!,g_r(h_r))$.
\end{thm}

\begin{proof}
We first derive the first statement from the second one. Given $\p$,
let $h$ and $g_1,\ldots\!,g_r\neq0$ as in the second statement such
that $\p=(g_1(h_1),\ldots\!,g_r(h_r))$. Applying $h^{-1}$ to $\p$,
we may even assume that $\p=(g_1(X_1),\ldots\!,g_r(X_r))$. For $i
\in \{1,\ldots\!,r\}$ we write $g_i(X_i)=X_i^{e_i}\tilde{g_i}(X_i)$
with $\tilde{g_i} \in \F_q[X_i]$, $\tilde{g_i}(0)\neq0$ and $e_i>0$
(note that $g_i\in\F_q[X_i]^{(q)}$, so indeed $g_i\in(X_i)$). Since
$g_0(0)=0$ for all $g_0 \in \p$, we must have
$\tilde{g_i}(X_i)\notin\p$, whence $X_i \in\p$ (since $\p$ is a
prime ideal). Substituting $X_j:=0$ for all $j\neq i$, we obtain
$X_i\in(X_i^{e_i}\tilde{g_i}(X_i))$. This implies that $e_i=1$ and
$\tilde{g_i}(X_i) \in \F_q^*$. Consequently,
$\p=(X_1,\ldots\!,X_r)$.

Now we prove the second statement. First note that $\a$ is generated by
{\it finitely many} linearized polynomials. Namely, $\a$ is generated by
finitely many general polynomials (since $\a$ is an ideal in a Noetherian ring),
and each of these general polynomials can be written as an $\F_q[X]$-linear
combination of finitely many of the linearized polynomials that generate $\a$. 
These together form the announced finite generating set.

So let $\a=(f_1,\ldots\!,f_m)$ for some $m \in \N$ and
$f_1,\ldots\!,f_m\in\F_q[X]^{(q)}$.
By Theorem~\ref{diagonalize}, there exist $h\in\TA_n(\F_q)^{(q)}$ and
$\tilde{h}\in\TA_m(\F_q)^{(q)}$ such that $g:=\tilde{h}fh^{-1}$ has the form
$g=(g_1,\ldots\!,g_m)$, where $g_i \in \F_q[X_i]^{(q)}\,$ for all
$i$ (and $g_i=0$ if $m>n$ and $n<i\leq m$).
Modifying $h$ and $\tilde{h}$ by a suitable permutation of the
variables, we may assume that $g_1,\ldots\!,g_r\neq0$ and
$g_{r+1}=\cdots=g_m=0$ for some $0\leq r\leq \min\{m,n\}$.
Since
\begin{eqnarray*}
(\tilde{h}_1(f),\ldots\!,\tilde{h}_m(f))\ \ =\ \ ((\tilde{h}f)_1,\ldots\!,(\tilde{h}f)_m) & = &
((gh)_1,\ldots\!,(gh)_m)\\
 & = & (g_1(h_1),\ldots\!,g_r(h_r))
\end{eqnarray*}
we are done as soon as we show that $\a=(\tilde{h}_1(f),\ldots\!,\tilde{h}_m(f))$.
Well then, we have $\tilde{h}_i(0)=0$ for all $i$, whence
$(\tilde{h}_1(f),\ldots\!,\tilde{h}_m(f))\subseteq(f_1,\ldots\!,f_m)$.
Likewise,
\begin{eqnarray*}
(f_1,\ldots\!,f_m) & = & \Big(\big(\tilde{h}^{-1}\big)_1\big(\tilde{h}_1(f),\ldots\!,\tilde{h}_m(f)\big),\ldots\!,
\big(\tilde{h}^{-1}\big)_m\big(\tilde{h}_1(f),\ldots\!,\tilde{h}_m(f)\big)\Big)\\
 & \subseteq & (\tilde{h}_1(f),\ldots\!,\tilde{h}_m(f))
\end{eqnarray*}
and thus $\a=(f_1,\ldots\!,f_m)=(\tilde{h}_1(f),\ldots\!,\tilde{h}_m(f))$.
\end{proof}\

\begin{cor} \label{PRRP}
Let $A=\F_q[X]/I$ be a finitely generated $\F_q$-algebra, where $I$ is an ideal
in $\F_q[X]$ generated by linearized polynomials. Then $A$ is (isomorphic to)
a polynomial ring over $\F_q$ if and only if $A$ is a domain.
\end{cor}\

\noindent Theorem~\ref{AMS} relates the Polynomial Ring Recognition
Problem to the Coordinate Recognition Problem for the case of two
variables. But this connection is in fact more general. Namely, it
is easily seen, that if $f_1 \in K[X_1,\ldots\!,X_n]$ is a coordinate,
then the $K$-algebra $K[X_1,\ldots\!,X_n]/(f_1)$ is a polynomial ring
over $K$ in $n-1$ variables. The reverse statement is the
Abhyankar-Sathaye Conjecture, which in case $n=2$ has an affirmative
answer by Theorem~\ref{AMS}. Although the Abhyankar-Sathaye Conjecture
is false in nonzero characteristic in general (as shown in Example~\ref{AMS-p}),
the statement holds for linearized polynomials:\\

\begin{thm}\label{ASCq}
If $f_1 \in \F_q[X]^{(q)}$ satisfies $\,\F_q[X]/(f_1) \cong_{\,\F_q}\F_q[Y_1,\ldots\!,Y_{n-1}]$, then $f_1$ is a coordinate.
\end{thm}

\begin{proof}
According to Corollary~\ref{polycoord}, $f_1=g_1(h_1)$, where $h_1$ is a
coordinate in $\F_q[X]^{(q)}$ and $g_1 \in \F_q[X_1]^{(q)}$. Then
$g_1(0)=0$, so $h_1$ divides $f_1$. Additionally, $(f_1)$ is a prime
ideal (as $\F_q[Y_1,\ldots\!,Y_{n-1}]$ is a domain), whence $f_1=ch_1$ for some $c
\in \F_q^*$. Thus, $f_1$ is a coordinate.
\end{proof}\

\subsection{Linearization Conjecture}\

\noindent The Linearization Conjecture doesn't hold in general in positive characteristic,
which is demonstrated in the following example. Throughout this section, 
$X$ (and also $Y$) denotes one variable.\\

\begin{example}
$f:=(X+Y^2,Y)\in\GA_2(\F_2)$ has order 2, but is not linearizable.
This already follows from two obvious facts about $f$: its linear
part equals the identity, and $f(0)=0$. Namely, suppose $g\in\GA_2(\F_2)$ 
such that $gfg^{-1}=l\in\GL_2(\F_2)$, and let $c:=g(0)$.
Then $\tilde{g}:=(X-c_1,Y-c_2)\circ g$ satisfies $\tilde{g}(0)=0$, and
\begin{equation}\label{linearization}
\tilde{g}f\tilde{g}^{-1} = (X-c_1,Y-c_2)l(X+c_1,Y+c_2)
\end{equation}
Since $f$ and $\tilde{g}$ have zero constant part, we can find 
the linear part of the lefthandside of (\ref{linearization}) 
by composing the linear parts of the factors of this composition. 
Hence, the linear part of the lefthandside equals the identity. 
Looking at the righthandside of (\ref{linearization}), 
we conclude that $l=(X,Y)$. But then also $f=(X,Y)$, a contradiction.
\end{example}\

\noindent In view of this example, a question arises: is the
Linearization Conjecture true in nonzero characteristic if we
additionally assume that the characteristic doesn't divide the order
of the automorphism? For linearized polynomial maps, this question
has an affirmative answer (Corollary~\ref{linearizable}). Because of
Theorem~\ref{corresp2}, the proof of this fact involves matrices in
$\GL_n(K[t])$
satisfying a polynomial relation over $K$.\\

\begin{lemma}\label{K-relation}
Let $R$ be a domain containing a field $K$, such that $K$ is
integrally closed in $L$, the field of fractions of $R$.
Furthermore, let $h(X)\in R[X]$ be the characteristic polynomial of
a given $A \in \GL_n(R)$, and $g(X)\in L[X]$ the minimal polynomial
of $A$ over $L$. Suppose $f(A)=0$ for some $f(X)\in K[X]$.
Then also $g(X), h(X)\in K[X]$.
\end{lemma}

\begin{proof}
$g(X)$ divides $f(X)$ in $L[X]$. Let $L'$ be a splitting field of
$f$ over $L$. Since $f(X)\in K[X]$, the roots of $f$ in $L'$ (and in
particular those of $g$) are integral over $K$, whence the
coefficients of $g$ are too. Moreover, $h(X)$ has the same roots as $g(X)$,
so the coefficients of $h$ are integral over $K$ as well.
But $K$ is integrally closed in $L$, so $g(X), h(X)\in K[X]$.
\end{proof}\

\begin{prop} \label{coprime}
Suppose $A\in\GL_n(K[t])$ satisfies $f(A)=0$ for some $f\in K[X]$, $f\neq0$.
Furthermore, write $f=f_1\cdots f_r$, with $f_1,\ldots\!,f_r\in K[X]$ mutually coprime. 
Then $K[t]^n=\Ker(f_1(A))\oplus\cdots\oplus \Ker(f_n(A))$,
and $A$ is conjugate over $K[t]$ to a block diagonal matrix, with blocks
$A_1,\ldots\!,A_r$ satisfying $f_i(A_i)=0$ for all $i$.

Moreover, if $f$ is the minimal (resp.\!\! characteristic) polynomial of $A$,
and each $f_i$ is monic, then $f_i$ is the minimal (resp.\! characteristic) polynomial of $A_i$ for all $i$.
\end{prop}

\begin{proof}
Consider the ideals $\a_i:=(f_i)\subseteq K[X]$. Then $\a_i+\a_j=(1)$ for all $i\neq j$.
Note that the ideals $\hat{\a_i}:=\a_1\cdots\a_{i-1}\a_{i+1}\cdots\a_r$ satisfy
$$
(1) = \prod_{i<j}(\a_i+\a_j) \subseteq \hat{\a_1}+\cdots+\hat{\a_r}
$$
whence $\hat{\a_1}+\cdots+\hat{\a_r}=(1)$. The above inclusion can be
justified as follows: any term $\a_{k_1}\cdots\a_{k_m}$ in the product on the left
(with $m:=\genfrac{}{}{}{1}{1}{2}r(r-1)$) originates from choices
between the two terms in all factors $\a_i+\a_j$.
Any term $\a_{k_1}\cdots\a_{k_m}$ must contain at least $r-1$
of the $\a_i$. Namely, given any $\a_i$ and $\a_j$ with $i\neq j$, the
factor $\a_i+\a_j$ appears in the product, so at least one of the two
must appear in the mentioned term. Therefore, $\a_{k_1}\cdots\a_{k_m}\subseteq\hat{\a_i}$
for some $i$.

So let $g_1,\ldots\!,g_r\in K[X]$ such that $g_1\hat{f_1}+\cdots+g_r\hat{f_r}=1$,
where for $i=1,\ldots\!,r$, $\hat{f_i}:=f_1\cdots f_{i-1}f_{i+1}\cdots f_r$.
We now claim that $K[t]^n=V_1\oplus\cdots\oplus V_r$, where $V_i:=\Ker(f_i(A))$ for all $i$.
First, note that the $V_i$ are $A$-invariant $K[t]$-submodules, and that they are all
free modules, being submodules of a finite free module over a principal ideal domain.
Second, for any $v\in K[t]^n$ we have
$$
v\,=\,Iv\,=\,g_1(A)\hat{f_1}(A)v+\cdots+g_r(A)\hat{f_r}(A)v\ \in\ V_1+\cdots+V_r
$$
since $f_i(A)\hat{f_i}(A)v=f(A)v=0$ for all $i$. Finally,
to justify the direct sum notation, suppose $v_1+\cdots+v_r=0$ for
certain $v_1\in V_1,\ldots\!,v_r\in V_r$. Then each $v_i$ satisfies
$$
v_i=(g_1(A)\hat{f_1}(A)+\cdots+g_r(A)\hat{f_r}(A))v_i=g_i(A)\hat{f_i}(A)v_i=g_i(A)\hat{f_i}(A)(v_1+\cdots+v_r)=0
$$
Now, for all $i\in\{1,\ldots\!,n\}$, let $m_i$ be the rank of $V_i$ as
a free $K[t]$-module, and $A_i\in\GL_{m_i}(K[t])$ the matrix representation
of the restriction of $A$ to $V_i$, with respect to some basis of $V_i$.
Taking these $r$ bases together to form a new basis of $K[t]^n$, we see that
$A$ is conjugate over $K[t]$ to the block diagonal matrix $A_0$ with
$A_1,\ldots\!,A_r$ on the diagonal. Also, $f_i(A_i)=0$ since $f_i(A)=0$ on $V_i$.

Now assume that each $f_i$ is monic. It is obvious from the shape of $A_0$ 
that the characteristic polynomial of $A_0$ (which is also the characteristic 
polynomial of $A$) is equal to the product of the characteristic polynomials of the $A_i$.
Also, the characteristic polynomial of $A_i$ (an element of $K[X]$ by Lemma~\ref{K-relation})
must be a power of the same monic irreducible polynomial that $f_i$ is also a power of.
Hence, if $f$ is the characteristic polynomial of $A$, then $f_i$ is
the characteristic polynomial of $A_i$.

Finally, assume that $f$ is the minimal polynomial of $A$ (which is also the
minimal polynomial of $A_0$). Choose $j\in\{1,\ldots\!,r\}$.
Suppose $h(A_j)=0$ for some $h(X)\in K[X]$, and define
$\hat{f}:=f_1\cdots f_{j-1}hf_{j+1}\cdots f_r$.
Then $\hat{f}(A_0)=0$, since it is the block diagonal matrix consisting of
the blocks $\hat{f}(A_i)$. (And $f_i(A_i)=0$ if $i\neq j$, and $h(A_i)=0$ if $i=j$.)
Whence, $f(X)\mid \hat{f}(X)$, {\it i.e.} $f_i(X)\mid h(X)$.
So $f_i$ must be the minimal polynomial of $A_i$.
\end{proof}\

\begin{thm} \label{irreduc}
Let $A \in \GL_n(K[t])$ such that its minimal polynomial $g(X)$ over $K(t)$
is an irreducible polynomial in $K[X]$ of degree $d\geq1$.
\begin{enumerate}
\item If $g$ is separable over $K$, then $A$ is conjugate
(over $K[t]$) to the $n\times n$ block diagonal matrix
where each block is the companion matrix of $g$.
\item If $d=n$ then $A$ is conjugate (over $K[t]$) to the companion matrix of $g$.
\end{enumerate}
\end{thm}\

\begin{proof}
The characteristic polynomial of $A$ (an element of $K[X]$ by Lemma~\ref{K-relation})
must be a power of $g$, say $g^m$ with $m\in\N^*$ such that $n=dm$.
Write $g(X)=X^d+c_{d-1}X^{d-1}+\cdots+c_1X+c_0$, where $c_i\in K$ for all $i$.
Moreover, let $L$ denote the splitting field of $g$ over $K$.
Also, we use the following notation: if $K_1\subseteq K_2$ are fields and $M\in\M_n(K_1[t])$,
then $\Ker_{K_2}(M)$ denotes the kernel of the endomorphism of $K_2[t]^n$
induced by $M$. This kernel is then viewed as a $K_2[t]$-module.
Furthermore, $M^\top$ denotes the transpose of any matrix $M$.

First, assume that $g$ is separable over $K$. Then $g$ has $d$ distinct roots in $L$.
Furthermore, $L/K$ is a Galois extension, say with Galois group $G$.
Since $L$ is the splitting field of an irreducible polynomial over $K$, $G$ acts transitively
on the roots of $g$. Therefore, we can find $\sigma_1,\sigma_2,\ldots\!,\sigma_d\in G$
(with $\sigma_1$ the identity map) and $\alpha\in L$ such that
$\sigma_1(\alpha),\ldots\!,\sigma_d(\alpha)$ are the roots of $g$ in $L$.
Then $\Ker_L(A-\sigma_i(\alpha)I)=\tilde{\sigma_i}(\Ker_L(A-\alpha I))$, where the automorphism
$\tilde{\sigma_i}$ is the natural extension of $\sigma_i$ to $L[t]^n$
(preserving $t$). As a result, $\Ker_L(A-\sigma_1(\alpha)I),\ldots\!,\Ker_L(A-\sigma_d(\alpha)I)$
all have the same rank as free $L[t]$-modules. (Note that indeed they are all
free modules, being submodules of a finite free module over a principal ideal domain.)
Moreover, from Proposition~\ref{coprime} (over $L[t]$ instead of $K[t]$)
we learn that $L[t]^n=\Ker_L(A-\sigma_1(\alpha)I)\oplus\cdots\oplus\Ker_L(A-\sigma_d(\alpha)I)$.
Consequently, the rank of $\Ker_L(A-\sigma_i(\alpha)I)$ equals $m$ for all $i$.

Again by Proposition~\ref{coprime} (and using the fact that $g$ is separable over $K$),
we know that $\Ker_{K(\alpha)}(A-\alpha I)$ is a direct summand of $K(\alpha)[t]^n$.
Also, tensoring with a free (and thus flat) module preserves kernels, so we have
$L\otimes_{K(\alpha)}\Ker_{K(\alpha)}(A-\alpha I)=\Ker_L(A-\alpha I)$. Hence,
since $\Ker_{K(\alpha)}(A-\alpha I)$ is a free $K(\alpha)[t]$-module,
its rank over $K(\alpha)[t]$ is equal to the rank of
$\Ker_L(A-\alpha I)$ over $L[t]$, which is $m$.

Let $\{v_1,\ldots\!,v_m\}$ be a basis of $\Ker_{K(\alpha)}(A-\alpha I)$.
Let $B\in\M_{n\times m}(K(\alpha)[t])$ be the matrix with $v_1,\ldots\!,v_m$ as its columns,
which satisfies $AB=\alpha B$. Note that then
$$
\M_n(K(\alpha)[t])B = \sum_{i=0}^{d-1} \M_n(K[t])\alpha^iB = \sum_{i=0}^{d-1} \M_n(K[t])A^iB \subseteq \M_n(K[t])B
$$
whence $\M_n(K(\alpha)[t])B = \M_n(K[t])B$.
Since $v_1,\ldots\!,v_m$ are the first $m$ elements of a basis of $K(\alpha)[t]^n$,
$B$ can be completed to an invertible $n\times n$ matrix over $K(\alpha)[t]$.
Taking together the first $m$ rows of its inverse, we obtain a
$B'\in\M_{m\times n}(K(\alpha)[t])$ such that $B'B=I_m$. Now define
$$
E_\alpha :=
\left(\begin{array}{cccc}
e_\alpha & 0 & \cdots & 0\\
0 & \ddots & \ddots & \vdots\\
\vdots & \ddots & \ddots & 0\\
0 & \cdots & 0 & e_\alpha
\end{array}\right)\ \in\ \M_{n\times m}(K(\alpha))
$$
where $e_\alpha:=(1\ \alpha\,\cdots\,\alpha^{d-1})^\top$ and each ``$0$'' is a column
consisting of $d$ zeroes. For every $n'\geq1$ this gives an isomorphism of $K[t]$-modules
\begin{eqnarray*}
\M_{n'\times n}(K[t]) & \longrightarrow & \M_{n'\times m}(K(\alpha)[t])\\
N & \mapsto & NE_\alpha
\end{eqnarray*}
using the fact that $\{1,\alpha,\ldots\!,\alpha^{d-1}\}$ is a $K[t]$-basis of $K(\alpha)[t]$.
In particular, there exists a $D \in \M_n(K[t])$ such that $DE_\alpha=B$. But we claim that
even $D \in \GL_n(K[t])$. Namely, $E_\alpha=(E_\alpha B')B\in\M_n(K(\alpha)[t])B = \M_n(K[t])B$,
say $E_\alpha=D'B$ with $D'\in\M_n(K[t])$. Then $D'DE_\alpha=D'B=E_\alpha$, so $D'D=I_n$.
As a result, $D'A(D')^{-1}E_\alpha=D'AB=D'\alpha B=\alpha D'B=\alpha E_\alpha$.
It is also readily verified that $C^\top e_\alpha=\alpha e_\alpha$, where
$$
C := \left(\begin{array}{ccccc} 0 & \cdots & \cdots & 0 & -c_0\\
1 & \ddots & & \vdots & \vdots\\
0 & \ddots & \ddots & \vdots & \vdots\\
\vdots & \ddots & \ddots & 0 & \vdots\\
0 & \hdots & 0 & 1 & \!\!\!-c_{\scriptscriptstyle d-1}\end{array}\right)
$$
is the companion matrix of $g$. Hence,
$$
D'A(D')^{-1} =
\left(\begin{array}{cccc}
C^\top & 0 & \cdots & 0\\
0 & \ddots & \ddots & \vdots\\
\vdots & \ddots & \ddots & 0\\
0 & \cdots & 0 & C^\top
\end{array}\right)
$$
Note that if in all of the above we replace $A$ by $C$ (so then $m=1$),
we obtain a proof of the fact that $C$ is conjugate to $C^\top$.
Combined with the above, this establishes the first statement of this theorem.

Now we turn to the second statement. To explain why we don't need separability
in this case, note that in the proof of the first statement we only used the fact
that the rank of $\Ker_{K(\alpha)}(A-\alpha I)$ is {\it at least} $m$.
So if in the second case we can show directly that the rank is at least 1,
we are done by copying the remainder of the proof of the first statement (with $m=1$).

We will now show that $\Ker_{K(\alpha)}(A-\alpha I)\neq\{0\}$ (which proves that
the rank is at least 1). Since $g(A)=(A-\alpha I)h(A)$ for some $h(X)\in K(\alpha)[X]$, 
$\Ker_{K(\alpha)}(A-\alpha I)$ contains the image of $h(A)$. So it suffices to show that $h(A)\neq0$.
To see this, note that $h(X)=\sum_{i=0}^{d-1}h_i(X)\alpha^i$, where $h_0,\ldots\!,h_{d-1}\in K[X]$ 
all have degree strictly less than $d$. So $h_i(A)\neq0$ for all $i$, whence $h(A)\neq0$. 
\end{proof}\

\begin{rem}
In Theorem~\ref{irreduc} the assumption that the minimal polynomial is
irreducible (instead of the more general case of being a power of an irreducible polynomial),
is really necessary. Namely, suppose $A=I_n+tN$, where $N$
is any nonzero nilpotent matrix in $\M_n(K)$. Then $(A-I_n)^n=0$, so
the minimal polynomial of $A$ over $K(t)$ is a nontrivial power of $X-1$ (and thus separable).
However, $A$ is not conjugate to an element of $\GL_n(K)$:
for any $B\in\GL_n(K[t])$ we have $B^{-1}AB=I+tB^{-1}NB$, and $tB^{-1}NB\notin\M_n(K)$.
\end{rem}\

\begin{cor}\label{finite order}
Let $K$ be a field and $A \in \GL_n(K[t])$ satisfying $A^d=I_n$,
where $\Char(K)\nmid d$.
Then there exists a $B\in\GL_n(K[t])$ such that $B^{-1}AB\in\GL_n(K)$.
\end{cor}

\begin{proof}
Note that the minimal polynomial of $A$ over $K(t)$, say $g(X)$, is an element of
$K[X]$ by Lemma~\ref{K-relation}, and of course a factor of $X^d-1$.
Since $\Char(K)\nmid d$, $X^d-1$ and its derivative have no common
zero in an algebraic closure of $K$, so neither do $g$ and $g'$.
Hence, $g$ is a product of mutually coprime monic irreducible polynomials,
which are also separable. Using Proposition~\ref{coprime}, we may
reduce to the case that $g$ is irreducible and separable.
But this case is settled by Theorem~\ref{irreduc}.
\end{proof}\

\noindent Theorem~\ref{corresp2} now gives

\begin{cor}\label{linearizable}
If $d$ and $q$ are relatively prime and $f \in \GA_n(\F_q)^{(q)}$
has finite order $d$, then $f$ is linearizable.
\end{cor}\

\section{Polynomial maps without mixed terms}\label{nomixedterms}\

\noindent In this final section we study all problems and conjectures mentioned 
in the Introduction for the case of a polynomial (map) {\it without mixed terms}. 
$K$ will be a field, mostly of characteristic zero.\\

\begin{defn}
A polynomial $f_1\in K[X]$ is said to be {\it without mixed terms} if we have 
$f_1\in K[X_1]+\cdots+K[X_n]$. 
A polynomial map $(f_1,\ldots\!,f_n)\in K[X]^n$ ($K$ a field) is
{\it without mixed terms} if each of the $f_i$ is.
\end{defn}\

\noindent Linearized polynomial maps are examples of polynomial maps 
without mixed terms. But the properties of linearized polynomial maps 
are very different from those of polynomial maps without mixed terms 
over a zero characteristic field. Namely, we have the following
theorem. First, $\BA_n(K)$ is the subgroup of {\it triangular}
automorphisms, {\it i.e.} all automorphisms $f=(f_1,\ldots\!,f_n)$
with $f_i-a_iX_i\in K[X_{i+1},\ldots\!,X_n]$ and $a_i\in K^*$ for
all $i$. (The notation comes from the fact that
$\BA_n(K)\cap\GL_n(K)$ equals the Borel subgroup of $\GL_n(K)$.)
Furthermore, such an $f$ is called {\it unitriangular} if
$a_1=\cdots=a_n=1$. $\BA_n^{\scriptscriptstyle(1)}(K)$ will be the
subgroup of unitriangular automorphisms. \\

\begin{thm} \label{unitrMap}
Let $f\in\GA_n(K)$ without mixed terms, and assume further that 
its linear part equals the identity. If $K$ has characteristic zero, 
then there exists a permutation $\pi$ of the $X_i$ such that $\pi^{-1}f\pi$ is unitriangular.

Furthermore, if $K$ has characteristic $p>0$, then there exists a
permutation $\pi$ of the $X_i$ such that
$\pi^{-1}f\pi\in\BA_n^{\scriptscriptstyle(1)}(K)+(K\big[X_1^{{}^p}\big]+\cdots+K\big[X_n^{{}^p}\big])^{{}^n}$.
\end{thm}

\begin{proof}
The first statement is a direct consequence (using Jacobians) of
Theorem~\ref{unitrMat}, which considers certain matrices with
entries in $K[X]$. In characteristic $p>0$ we can use the same
theorem, but we need to take into account that the $i$th partial
derivative of a power $X_i^{{}^m}$ vanishes if and only if $p\mid m$.
\end{proof}\

\noindent Note that, given any automorphism without mixed terms, we can compose it
on the left with the inverse of its linear part, to obtain an automorphism
satisfying all hypotheses of Theorem~\ref{unitrMap}.\\

\begin{defn}
$A=(a_{ij})\in\M_n(K[X])$ is a {\it matrix in separated variables}
if $a_{ij}\in K[X_j]$ for all $i$ and $j$. These matrices form a
left $\M_n(K)$-submodule of $\M_n(K[X])$. 
\end{defn}\

\noindent In the following, we use some well-known terminology 
from matrix theory: A {\it principal submatrix (of order $k$)} 
of a square matrix is a submatrix formed by a subset of ($k$)
rows and the corresponding subset of columns. And a {\it principal
($k$-)minor} of a square matrix is the determinant of a principal
submatrix (of order $k$). \\

\begin{thm}\label{unitrMat}
Every matrix $A\in\GL_n(K[X])$ in separated variables with $A(0)=I_n$ is 
(after conjugation by a permutation matrix) unitriangular (upper triangular with only 1's on the diagonal).
\end{thm}

\begin{proof}
By Lemma~\ref{prinmin1}, we are done if we can prove that all
principal minors of $A$ are equal to 1. First, note that $|A|\in
K^*$ and $|A(0)|=1$ together imply that $|A|=1$. For all $1\leq
j\leq n$, let $A_j$ be the matrix obtained from $A$ by deleting its
$j$th row and column. Note that $A_j\in\M_{n-1}(K[\hat{X}_j])$ is a 
matrix in seperated variables satisfying $A_j(0)=I_{n-1}$. 
Moreover, expanding the determinant of $A$ along its $j$th column and 
substituting $X_j=0$, we obtain $1=|A_{|_{X_j=0}}|=a_{jj}(0)\cdot|A_j|=|A_j|$ 
($A(0)=I_n$, so $a_{ij}(0)=0$ whenever $i\neq j$).

From all this we may conclude that for every $A\in\GL_n(K[X])$ in separated variables 
satisfying $A(0)=I_n$, we have $|A|=1$, each $A_j$ is a matrix in 
$\GL_{n-1}(K[\hat{X}_j])$ in separated variables and $A_j(0)=I_{n-1}$. 
Induction now proves that for every matrix $A\in\GL_n(K[X])$ in separated variables 
satisfying $A(0)=I_n$, all principal minors are equal to 1.
\end{proof}\

\begin{rem}
The proof of the above theorem in particular implies that all
diagonal elements of $A$ (being principal minors) are equal
to 1. But this can also be proved directly. Namely, since
$A(0)=I_n$, each non-diagonal entry $a_{ij}$ satisfies $X_j\mid
a_{ij}$. The fact that $A\in\GL_n(K[X])$ implies that
$(a_{i1}(X_1),\ldots\!,a_{in}(X_n))=(1)$ in $K[X]$ for all $i$.
Substituting $X_j=0$ for all $j\neq i$, we obtain $a_{ii}\in K^*$.
But $A(0)=I_n$, whence $a_{11}=\cdots=a_{nn}=1$.
\end{rem}\

\noindent Additionally, Theorem~\ref{unitrMat} partly solves the
Jacobian Conjecture:

\begin{cor}\label{JCWMT}
The Jacobian Conjecture is satisfied for polynomial maps without
mixed terms.
\end{cor}

\begin{proof}
 If $f$ is a polynomial map without mixed terms satisfying $|\J f|\in K^*$, 
then also $|\J f(0)|\in K^*$, {\it i.e.} $f$ has invertible linear part. 
Composing $f$ on the left with the inverse of its linear part, 
we may assume that $\J f(0)=I_n$. According to Theorem~\ref{unitrMat}, 
this means that $f$ is unitriangular after a permutation of the variables. 
\end{proof}\

\begin{lemma}\label{prinmin1}
Let $R$ be a domain. Suppose $A=(a_{ij})\in\GL_n(R)$ has the
property that all its principal minors are equal to 1. Then $A$ is 
(after conjugation by a permutation matrix) unitriangular. 
\end{lemma}

\begin{proof}
We may assume that $R$ is a field. Note that if all principal minors of 
a matrix equal 1, then any principal submatrix also has this property. 
Further, a column of a square matrix is called an {\it elementary column} 
if its diagonal entry equals 1 and all its remaining entries are 0.
Note that the property of having an elementary column is invariant \
under conjugation by a permutation matrix. 
(Partly due to the fact that conjugation by a permutation
matrix permutes the diagonal elements.) 

We will prove the theorem by induction on $n$. It is trivial for
$n=1$. If $n=2$ then $|A|=1$ implies $a_{12}a_{21}=0$, which also
settles this case (as $R$ is a field). So we will assume from now on
that $n\geq3$ and that the statement holds in lower dimensions.
For all $1\leq i\leq n$, let $A_i$ be the matrix obtained from $A$
by deleting its $i$th row and column. Note that we may apply the
induction hypothesis to $A_i$. 

We are done if $A$ contains an elementary column: 
if this is the case, we may (after permutation) assume that 
the first column is elementary, and then apply the induction 
hypothesis to $A_1$ to obtain (after permutation) a unitriangular matrix. 

Now we assume that $A$ doesn't have an elementary column,
and aim to arrive at a contradiction. Take $i\in\{1,\ldots\!,n\}$. 
By the induction hypothesis, $A_i$ contains an elementary column. 
So there is a $j\neq i$ such that the $j$th column of $A$ is ``almost elementary'', 
{\it i.e.}  $a_{jj}=1$ and $a_{kj}=0$ for $k\notin\{i,j\}$. 
And $a_{ij}\neq0$, as $A$ has no elementary column. Associating a $j$ 
to each $i$ in this way, we obtain a map $\sigma$ from $\{1,\ldots\!,n\}$ to itself. 
$\sigma$ is obviously injective, and thus a permutation. Hence, 
$a_{ij}=0$ for all $i$ and $j$ with $j\notin\{i,\sigma(i)\}$ 
(and $a_{ii}=1$ for all $i$). 

Using the induction hypothesis on $A_n$ again, we may assume 
(after conjugation by a permutation matrix) that $A_n$ is unitriangular. 
Hence, $\sigma(i)>i$ for all $i<n$. But then we must have $\sigma(n)=1$ 
and $\sigma(i)=i+1$ for all $i<n$. Hence, expanding the determinant of $A$ 
along the $n$th row we obtain $0=|A|-1=a_{1\sigma(1)}\cdots a_{n\sigma(n)}$, 
which contradicts the fact that all $a_{i\sigma(i)}$ are nonzero.
\end{proof}\

\begin{rem}
For a domain $R$ and any $A'\in\M_n(R)$, Corollary~6.3.9 in \cite{Essen}
gives a result which is very similar to Lemma~\ref{prinmin1}. It says that
if every principal minor of $A'$ is equal to 0, then $A'$ can be conjugated 
by a permutation matrix such that the resulting matrix is an upper
triangular matrix with zero diagonal. This result and Lemma~\ref{prinmin1}
are actually easily shown to be equivalent!

Namely, we can use the well-known fact that the coefficient of $X^{n-k}$ in the
characteristic polynomial $P_{\boldsymbol{\cdot}} (X)$ of an $n\times n$-matrix equals $(-1)^k$ times the sum
of all principal $k$-minors. So suppose $A\in\GL_n(R)$ is such that all its
principal minors are equal to 1. Then any principal submatrix $A'_0$ of $A':=A-I_n$
is of the form $A'_0=A_0-I$, where $I$ is the identity matrix of the corresponding size,
and $A_0$ is the principal submatrix of $A$ consisting of the corresponding rows and
columns. Let $m$ be the number of rows (or columns) of $A_0$.
Since all principal minors of $A_0$ are equal to 1, and for each $k$ there are
$\binom{m}{k}$ principal $k$-minors, $P_{A_0}(X)=X^m-mX^{m-1}+\binom{m}{2}X^{m-2}-\cdots+(-1)^m=(X-1)^m$.
But then $P_{A'_0}(X)=|XI_m-A'_0|=|(X+1)I_m-A_0|=P_{A_0}(X+1)=X^m$. So $A'_0$ is
nilpotent, and in particular $|A'_0|=0$. Now that every principal minor of $A'$
is equal to 0, the result in \cite{Essen} gives a permutation matrix $B$ such that
$B^{-1}AB=B^{-1}A'B+I_n$ is upper unitriangular. Similarly, we can obtain the result in
\cite{Essen} from our Lemma~\ref{prinmin1}.
\end{rem}\

\noindent Now we consider the remaining problems and conjectures presented in 
the Introduction. First, the Tame Generators Problem: an immediate consequence of 
Theorem~\ref{unitrMap}. (Triangular automorphisms are obviously tame.)\\

\begin{cor}
 Over a characteristic zero field, all invertible polynomial maps without mixed terms are tame.
\end{cor}\

\noindent Also, we can use Theorem~\ref{unitrMap} to partly solve the Linearization Conjecture 
(Corollary~\ref{linconnonmixed}). 
It is unknown to the author whether this conjecture also holds for the most general form 
of an invertible polynomial map without mixed terms.\\

\begin{lemma}
Let $f=(aX_1+p,g)\in\GA_n(K)$, where $a\in K^*$, $p\in K[X_2,\ldots\!,X_n]$ and 
$g\in\GA_{n-1}(K)$ (in the variables $X_2,\ldots\!,X_n$). 
Suppose $f$ has finite order. Then $h^{-1}fh=(aX_1,g)$ for some $h\in\EA_n(K)$ 
with $h(X_i)=X_i$ for $i\geq2$.

In particular, the Linearization Conjecture holds for triangular maps.
\end{lemma}

\begin{proof}
The second statement follows by repeatedly applying the first one to a 
given triangular map. So let $f=(aX_1+p,g)$ be as described, and suppose 
it has finite order $d\geq1$. One readily verifies that for all $k\geq1$, 
$f^k$ has the form $(a^kX_1+p_k,g^k)$, where $p_k\in K[X_2,\ldots\!,X_n]$ (and $p_d=0$). 
From $f^{k+1}=f^k\circ f$ we get that $p_{k+1}=p_k(g)+a^kp$ for all $k$. 

Now let $q:=\sum_{k=1}^{d-1}\frac{1}{da^k}p_k$, and $h:=(X_1-q,X_2,\ldots\!,X_n)$. 
Then $h^{-1}fh=(aX_1,g)$ if and only if $-aq+p+q(g)=0$. The latter follows from 
the fact that $q(g)$ equals
$$
\sum_{k=1}^{d-1}\tfrac{1}{da^k}p_k(g)\ =\ \sum_{k=1}^{d-1}\tfrac{1}{da^k}(p_{k+1}-a^kp)\ =\ 
\sum_{m=2}^d\tfrac{1}{da^{m-1}}p_m-\tfrac{d-1}{d}p\ =\ aq-p
$$
using $p_1=p$ and $p_d=0$.
\end{proof}\

\begin{cor}\label{linconnonmixed}
Let $f$ be a polynomial map without mixed terms over a characteristic zero field, 
and suppose the matrix of its linear part is diagonal. 
Then the Linearization Conjecture holds for $f$.  
\end{cor}

\begin{proof}
 By Theorem~\ref{unitrMap}, we may assume that $f$ is triangular.
\end{proof}\

\noindent The next one (the Coordinate Recognition problem) is easy.\\

\begin{prop}\label{coordrecognnonmixed}
Let $\Char(K)=0$ and $f\in K[X]$ a polynomial without mixed terms, say $f=f_1+\cdots+f_n$ 
with $f_i\in K[X_i]$ for all $i$. Then $f$ is a coordinate iff at least one of the $f_i$ 
has degree 1. 
\end{prop}

\begin{proof}
A necessary condition for any polynomial in $K[X]$ to be a coordinate, 
is that the ideal of its partial derivatives is the unit ideal in $K[X]$ 
(as these partial derivatives form the first row of an invertible Jacobian matrix). 
In this case this condition is also sufficient, since it is here 
equivalent to saying that at least one of these partial derivatives is a nonzero constant 
(the partial derivatives cannot have a common zero in an algebraic closure of $K$). 
\end{proof}\

\noindent Unfortunately, the Polynomial Ring Recognition Problem (say for a 
finitely generated $K$-algebra $A=K[X]/I$, $I$ an ideal) is still unsolved 
if $\Char(K)=0$ and $A$ is at least three-generated over $K$, even if $I$ is generated 
by polynomials without mixed terms. In particular, we can finish this paper 
with the following question.\\

\begin{question}
 Do polynomials without mixed terms satisfy the Abhyankar-Sathaye Conjecture?
\end{question}\

\section*{Acknowledgement}\

\noindent The author is very grateful to Arno van den Essen and Stefan Maubach 
for useful discussions and comments.\\

\vspace{1cm}

\doublespacing

\begin{center}

\end{center}

\end{document}